\documentclass[12pt,a4paper]{amsart}
\usepackage[utf8]{inputenc}
\usepackage[english]{babel}
\usepackage{amsmath,amssymb,amsthm,amscd,graphicx,mathrsfs,url,dsfont,enumitem}
\usepackage[usenames,dvipsnames]{color}
\usepackage[colorlinks=true,linkcolor=Blue,citecolor=Green,urlcolor=Blue]{hyperref}
\usepackage{dsfont}
\usepackage[super]{nth}
\usepackage[open, openlevel=2, depth=3, atend]{bookmark}
\hypersetup{pdfstartview=XYZ}
\usepackage{epigraph}
\usepackage{mathtools}
\usepackage{enumitem}
\usepackage{mathrsfs}
\usepackage{comment}

\def\?[#1]{\textbf{[#1]}\marginpar{\Large{\textbf{??}}}}

\usepackage{import}
\usepackage{xifthen}
\usepackage{pdfpages}
\usepackage{transparent}
\usepackage{chngcntr}
\counterwithin{figure}{section}
\usepackage{subcaption}

\numberwithin{equation}{section}

\newtheorem{theorem}{Theorem}[section]
\newtheorem{lemma}[theorem]{Lemma}

\newtheorem{proposition}[theorem]{Proposition}

\theoremstyle{remark}
\newtheorem{remark}{Remark}
\newtheorem{example}{Example}

\theoremstyle{definition}

\newcommand{\kg}{\mathfrak g}
\newcommand{\C}{\mathbb C}
\renewcommand{\S}{\mathbb S}
\newcommand{\D}{\mathbb D}
\newcommand{\cA}{\mathcal A}
\renewcommand{\d}{\mathrm d}
\newcommand{\kB}{\mathfrak B}
\newcommand{\del}{\partial}

\newcommand{\cJ}{\mathcal J}
\newcommand{\R}{\mathbb R}
\newcommand{\cL}{\mathcal L}

\newcommand{\delbar}{\bar\partial}
\newcommand{\Z}{\mathbb Z}

\renewcommand{\Re}{\mathrm{Re}}

\newcommand{\cM}{\mathcal M}
\newcommand{\cQ}{\mathcal Q}
\newcommand{\bJ}{\mathbf{J}}
\newcommand{\Ad}{\mathrm{Ad}}
\newcommand{\cT}{\mathcal{T}}

\newcommand{\N}{\mathbb{N}}

\newcommand{\cH}{\mathcal{H}}

\newcommand{\tr}{\mathrm{tr}}

\newcommand{\cG}{\mathcal{G}}

\newcommand{\qand}{\quad\text{and}\quad}

\newcommand{\cZ}{\mathcal{Z}}

\makeatletter
\newsavebox{\@brx}
\newcommand{\llangle}[1][]{\savebox{\@brx}{\(\m@th{#1\langle}\)}%
  \mathopen{\copy\@brx\mkern2mu\kern-0.9\wd\@brx\usebox{\@brx}}}
\newcommand{\rrangle}[1][]{\savebox{\@brx}{\(\m@th{#1\rangle}\)}%
  \mathclose{\copy\@brx\mkern2mu\kern-0.9\wd\@brx\usebox{\@brx}}}
\makeatother

\title[Non-Abelian multiplicative chaos]{Non-Abelian multiplicative chaos on the circle}
\author{Guillaume Baverez}
\email{guillaume.baverez@bicmr.pku.edu.cn}
\address{Beijing International Center for Mathematical Research, Peking University}
\date{}


\begin{document}

\begin{abstract}
In this note, we introduce a generalisation of multiplicative chaos measures which is both non-Gaussian and non-Abelian. The renormalisation procedure takes as inputs an irreducible unitary representation of a compact connected Lie group, together with a Kac--Moody unitarising measure at some level $\kappa$, and outputs a random measure on the circle with values in the space of positive definite Hermitian endomorphisms of the representation space.

So far, our construction is valid for the range of $\kappa$-values corresponding to the $L^2$-phase. The proof follows the usual route in the theory of multiplicative chaos, relying on an exact formula for the one-point function and a bound on the two-point function at colliding points. These expressions are derived using the rich algebraic structure of the theory: namely, we establish a one-dimensional version of the Knizhnik--Zamolodchikov equations.
\end{abstract}

\maketitle


\section{Introduction}

    \subsection{Overview and main result}
The theory of Gaussian multiplicative chaos --- pioneered by Kahane in the context of turbulence \cite{Kahane85} --- gives a way to exponentiate a logarithmically correlated Gaussian distribution and has had a profound and lasting influence on probability theory, with applications ranging from random matrix theory to number theory and finance, and rather spectacularly to conformal field theory (see \cite{RV14-review,RV25-proceedings,Sheffield_proceedings,BP25} for reviews and additional references). Several generalisations of Kahane's original work have been proposed in different directions: complex-valued distributions \cite{LRV15,JSW20}, matrix-valued measures \cite{CRV13}, and exponentials of non-Gaussian $\log$-correlated fields \cite{Jego20,ABJL23}. 

This note sits at the intersection of the last two directions, with the construction of a multiplicative chaos based on the recently introduced Kac--Moody unitarising measures \cite{Bav26}. To set the stage (see Section \ref{subsubsec:lie-groups} for background on Lie groups), let $G$ be a compact connected Lie group with complexification $G_\C$ and dual Coxeter number $\check h$. The space of Hermitian Yang--Mills metrics in the trivial holomorphic bundle $\D\times G_\C$ is defined as
\[\cM_\D:=\{H\in C^\infty(\D;\exp(i\kg))|\,\delbar(H^{-1}\del H)=0=\del(\delbar HH^{-1})\text{ and }H(0)=1_{G_\C}\},\]
and the Kac--Moody unitarising measure at level $\kappa<-2\check h$ is a probability measure on $\cM_\D$, denoted $\nu_\kappa$. The main properties of this measure are recalled in Section \ref{subsubsec:KMU}, but let us already point out that its samples almost surely do not converge on $\S^1=\del\D$. The goal of this note is precisely to exhibit a renormalisation procedure which produces a sensible mathematical object on the boundary.


Let $\varrho:G_\C\to\mathrm{GL}(V)$ be the complexification of an irreducible unitary representation of $G$ on some complex (finite dimensional) Hilbert space $(V,\langle\cdot,\cdot\rangle_V)$, and let $\rho:\kg_\C\to\mathrm{End}(V)$ be the induced representation of the Lie algebra. For all $H\in\cM_\D$, $\varrho(H)$ takes values in the space $\cH_V$ of positive definite Hermitian endomorphisms of $V$. In Proposition \ref{prop:one-point}, it will be shown that $\int_{\cM_\D}\varrho(H(z))\d \nu_\kappa(H)=(1-|z|^2)^\frac{\Omega_\rho}{\kappa+\check h}\mathrm{Id}_V$ for all $z\in\D$, where $\Omega_\rho$ is the eigenvalue of the quadratic Casimir. This suggests the form of the renormalisation prefactor in our main result (see Section \ref{sec:preliminaries} for the different notions involved in the statement).

\begin{theorem}\label{thm:convergence}
   Fix $\varrho$ as above and $\kappa<\min(-2\check h,-\check h-2\lambda_\rho^+)$, with $\lambda_\rho^+$ defined in \eqref{eq:lambda_rho}.
   
    For each $\epsilon>0$, let $M_{\kappa,\rho}^\epsilon$ be the random $\cH_V$-valued Borel measure on $\S^1$ defined by 
\begin{equation}\label{eq:def-reg-measure}
\d M^\epsilon_{\kappa,\rho}(e^{i\theta}):=\,(1-e^{-2\epsilon})^\frac{-\Omega_\rho}{\kappa+\check h}\varrho\left(H(e^{-\epsilon+i\theta})\right)\d\theta,\qquad\qquad\theta\in\R,
\end{equation}
with $H$ sampled from $\nu_\kappa$. Then, the family $(M^\epsilon_{\kappa,\rho})_{\epsilon>0}$ converges weakly in probability as $\epsilon\to0$ to some non-trivial $\cH_V$-valued measure $M_{\kappa,\rho}$.
\end{theorem}

This result strengthens \cite[Theorem 1]{Bav26} as it constructs a sample space for the Kac--Moody unitarising measure which is intrinsic to the circle, rather than some space of holomorphic forms in the disc. The key observation is to map holomorphic forms to Hermitian Yang--Mills metrics, which converge to a sound (albeit rough) mathematical object as we renormalise when approaching the boundary. 

 The proof of Theorem \ref{thm:convergence} is a standard argument once we have an expression for the one- and two-point correlation functions. Remarkably, we can get explicit expressions using the rich algebraic structure of the theory: we will see that the correlation functions solve a one dimensional version of the Knizhnik--Zamolodchikov equations (see \cite[Chapter~15]{dFMS_CFT-book} or \cite{EFK_KZ-lectures} for comprehensive reviews). In particular, we recover the conformal weights of the WZW model \cite[eq. (15.87)]{dFMS_CFT-book}.\footnote{In the right-hand side of this equation, the numerator is the Casimir eigenvalue \cite[eq. (13.127)]{dFMS_CFT-book}, $g$ is the dual Coxeter number \cite[eq. (13.35)]{dFMS_CFT-book}, and $k$ is the level.} In an ongoing programme with Baptiste Cercl\'e and Antoine Jego, we plan to construct higher dimensional versions of this multiplicative chaos and put these emergent connections with conformal field theory and gauge theory on firmer ground.

    \subsection{Discussion}

Our strategy for the proof of Theorem \ref{thm:convergence} is based on second moment computations, and is valid for the range of $\kappa$-values known to experts in multiplicative chaos as the $L^2$-phase. Hence, the bound $\kappa<-\check h-2\lambda_\rho^+$ is suboptimal, and the value $\kappa=-\check h-2\lambda_\rho^+$ is not a true phase transition, but simply the point where the limit measure fails to admit second moments and the method of proof stops working. To the best of our knowledge, there is no model of matrix-valued multiplicative chaos available in the full $L^1$-phase at the time of writing. Indeed, we are only aware of the paper \cite{CRV13} which is also restricted to the $L^2$-phase. The authors consider general isotropic Gaussian fields and derive suitable bounds on the moments using results from random matrix theory. By contrast, we work with a very specific non-Gaussian model from which we exploit the Kac--Moody symmetry to derive explicit formulas.

Going beyond the $L^2$-phase typically requires understanding ``where the measure lives" \cite{RV14-review,Berestycki}, which in our case amounts to establishing bounds on the probability that the spectrum of $\varrho(H(z))$ is unusually large. More generally, one would like to know the multifractal properties of the measure, or (which is essentially the same) compute for each $\alpha\in\R$ the Hausdorff dimension of $\cT_\alpha:=\lbrace z\in\S^1|\,\lim_{\epsilon\to0}\frac{\log\mathrm{Tr}_V(\varrho(H(e^{-\epsilon}z))}{\log(1/\epsilon)}=\alpha\rbrace$.

One way to approach this type of question is via the Mellin transform of $\varrho(H)$: a back-of-the-envelope computation following Proposition \ref{prop:one-point} suggests that the quantity $\int_{\cM_\D}\mathrm{Tr}_V(\varrho(H(e^{-\epsilon}z))^s)\d\nu_\kappa(H)$ behaves like $\epsilon^\frac{s^2\Omega_\rho}{\kappa+s^2\check h}$ for complex $s$ as $\epsilon\to0$ with $z\in\S^1$ fixed, so one might expect a multifractal spectrum which is not a polynomial, but rather a rational function. This is in sharp contrast with the Gaussian case \cite[Section 4]{RV14-review}.

In a somewhat orthogonal direction, one could try to construct an ``imaginary" version of this multiplicative chaos, which would be an $\mathrm{End}(V)$-valued distribution on the circle. This point of view appears when one sees the Kac--Moody unitarising measure as a (formal) measure on the loop group of $G$ rather than Hermitian Yang--Mills metrics (this was the point of view initially adopted in \cite{Bav26}). However, due to the renormalisation, there is no chance that the limit measure will take values in the space of unitary operators on $V$; this is already true in the Abelian case since the standard imaginary chaos is a real-valued distribution \cite{LRV15,JSW20}, hence does not take values in $\S^1$.

\subsection*{Acknowledgements} We thank Nathana\"el Berestycki, Baptiste Cercl\'e, Antoine Jego, R\'emi Rhodes and Eero Saksman for helpful comments. This work was supported by the National Natural Science Foundation of China (Grant No. 12526204).

   \section{Preliminaries}\label{sec:preliminaries}

    \subsection{Lie groups}\label{subsubsec:lie-groups}
Here we recall some elementary facts on Lie groups and their representations, as can be found in any standard textbook such as \cite{Knapp}. See also \cite{EFK_KZ-lectures,dFMS_CFT-book} for CFT-oriented references. Our conventions and notations follow \cite[Section~2.1]{Bav26}.
    
We fix a compact connected Lie group $G$ with complexification $G_\C$, and we assume without loss of generality that it is semisimple. The respective Lie algebras are denoted $\kg$, $\kg_\C$. There is an antiholomorphic Cartan involution $\theta:G_\C\to G_\C$ with fixed point set $G$, which is a group automorphism of $G_\C$ (viewed as a real Lie group). 
Moreover, the map $\kg\times G\to G_\C,\,(x,g)\mapsto\exp(ix)g$ is a diffeomorphism \cite[Theorem 6.31]{Knapp}, giving us the identification $G_\C/G\simeq\exp(i\kg)$. We will also write $g^*:=\theta(g^{-1})$ for $g\in G_\C$, which coincides with the Hermitian adjoint in the case $G_\C=\mathrm{SL}_N(\C)$. At the level of the Lie algebra, the Cartan involution gives a Lie algebra automorphism with fixed set $\kg$, which we keep denoting $\theta$ by a slight (but common) abuse of notation. Given $\alpha\in\kg_\C$, we write $\alpha^*:=-\theta(\alpha)$.

Recall that $G_\C$ acts on itself by conjugation, leading to the adjoint representation of $G_\C$ on $\kg_\C$ by differentiating at the identity, i.e. $\mathrm{Ad}_g(\alpha)=\frac{\d}{\d\epsilon}_{|\epsilon=0}ge^{\epsilon\alpha}g^{-1}$ for all $g\in G_\C$ and $\alpha\in\kg_\C$. We will frequently use the shorthand $\Ad_g(\alpha)=g\alpha g^{-1}$. By differentiating again at $g=1_{G_\C}$, one gets the adjoint representation at the Lie algebra level $\mathrm{ad}:\kg_\C\to\mathrm{End}(\kg_\C)$, which is defined by $\mathrm{ad}_x(\alpha)=[x,\alpha]$ for all $x,\alpha\in\kg_\C$.

Let $\rho$ be the complexification of a unitary representation of $\kg$ on some finite dimensional complex Hilbert space $(V,\langle\cdot,\cdot\rangle_V)$. Then, $\rho$ induces an invariant symmetric bilinear form $\kappa_\rho$ on $\kg_\C$ defined by $\kappa_\rho(x,y)=\mathrm{Tr}_V(\rho(x)\circ\rho(y))$. All these bilinear forms are proportional to each other and non-degenerate for semisimple $\kg_\C$, and we let $\tr$ be the one which is normalised so that the longest root has squared length~$2$. Once and for all, we fix $\kB\subset i\kg$ an orthonormal basis for $\tr$, i.e. $\tr(a,b)=\delta_{a,b}$ for all $a,b\in\kB$. The quadratic Casimir operator of a representation $\rho$ is $\Omega_\rho:=\sum_{b\in\kB}\rho(b)\circ\rho(b)$. If $\rho$ is irreducible, $\Omega_\rho$ is proportional to the identity, and we will freely identify the operator with its eigenvalue. In particular, one can define the dual Coxeter number by the formula $\check h=\frac{1}{2}\Omega_\mathrm{ad}$.  The next lemma records an elementary formula for future reference (we use the same notation for the commutator in $\mathrm{End}(V)$ and the Lie bracket in $\kg_\C$).
\begin{lemma}\label{lem:coxeter}
    For all $\alpha\in\kg_\C$, we have $\sum_{b\in\kB}[\rho(\alpha),\rho(b)]\circ\rho(b)=\check h\rho(\alpha)$.
\end{lemma}
\begin{proof}
    Write $A=\sum_{b\in\kB}[\rho(\alpha),\rho(b)]\circ\rho(b)$ and $B=\sum_{b\in\kB}\rho(b)\circ[\rho(\alpha),\rho(b)]$. On the one hand, $A-B=[\rho(\alpha),\Omega_\rho]=0$ since the Casimir operator is central. On the other hand, $A+B=\sum_{b\in\kB}[[\rho(\alpha),\rho(b)],\rho(b)]=\rho(\Omega_\mathrm{ad}(\alpha))=2\check h\rho(\alpha)$. Hence, $A=\check h\rho(\alpha)$ as required.
\end{proof}

By definition, the tensor product representation on $V\otimes V$ is given by $\rho\otimes\mathrm{Id}_V+\mathrm{Id}_V\otimes\rho$. The quadratic Casimir of the product representation is
\[\Omega_{\rho\otimes\mathrm{Id}_V+\mathrm{Id}_V\otimes\rho}=\sum_{b\in\kB}(\rho(b)\otimes\mathrm{Id}_V+\mathrm{Id}_V\otimes\rho(b))^2=\sum_{b\in\kB}2\rho(b)\otimes\rho(b)+2\Omega_\rho\mathrm{Id}_{V\otimes V}.\]
The decomposition of a tensor product representation as a direct sum of irreducibles is known as the Clebsch--Gordan problem: as representations of $\kg_\C$, one has $V\otimes V\simeq\oplus_{\rho'}m_{\rho,\rho'}V_{\rho'}$, where the sum is over the irreducible unitary representations of $\kg_\C$, each representation $V_{\rho'}$ appearing with multiplicity $m_{\rho,\rho'}\geq0$. Hence, the operator $\sum_{b\in\kB}\rho(b)\otimes\rho(b)$ acts as a multiple of the identity in each $V_{\rho'}\hookrightarrow V\otimes V$, with eigenvalue $\frac{1}{2}\Omega_{\rho'}-\Omega_\rho$. In particular, the spectrum of this operator is bounded above and below respectively by
  \begin{equation}\label{eq:lambda_rho}
  \lambda_\rho^+:=\frac{1}{2}\max\{\Omega_{\rho'}|\,m_{\rho,\rho'}\neq0\}-\Omega_\rho\qand\lambda_\rho^-:=\frac{1}{2}\min\{\Omega_{\rho'}|\,m_{\rho,\rho'}\neq0\}-\Omega_\rho.
  \end{equation}
  The operator $\sum_{b\in\kB}\rho(b)\otimes\rho(b)$ is fundamental in the theory of the Knizhnik--Zamolodchikov equations; see \cite[Lecture~2.5]{EFK_KZ-lectures} for additional details.

    \subsection{Hermitian Yang--Mills metrics}\label{subsec:HYM}
Let $DG_\C$ be the group of holomorphic $G_\C$-valued functions on $\D$, and 
\[D_0G_\C:=\left\lbrace g\in DG_\C|\,g(0)=1_{G_\C}\right\rbrace.\]
Following \cite{Donaldson}, a Hermitian Yang--Mills metric in the trivial holomorphic bundle $\D\times G_\C$ is a smooth function $H:\D\to\exp(i\kg)$ satisfying the $2d$ Yang--Mills equations
\begin{equation}\label{eq:hym}
\delbar(H^{-1}\del H)=0\qand\del(\delbar HH^{-1})=0.
\end{equation}
However, contrary to Donaldson, we do not impose any kind of smoothness up to the boundary. The Yang--Mills equations \eqref{eq:hym} imply the factorisation $H=g^*g$ for some $g\in DG_\C$ \cite[Section 3.3]{Donaldson}, which is unique up to left multiplication by $G$. We define
\[\cM_\D:=\left\lbrace H\in C^\infty(\D;\exp(i\kg))\text{ Hermitian Yang--Mills with }\,H(0)=1_{G_\C}\right\rbrace.\]
For $H\in\cM_\D$, the above factorisation is then unique if we impose $g(0)=1_{G_\C}$. To sum up, the action of $D_0G_\C$ on $\cM_\D$ defined by $(H,g)\mapsto g^*Hg$ is both free and transitive. For what follows, a better point of view is to observe that the group $\{g\in DG_\C|\,g(0)\in G\}$ also acts transitively on $\cM_\D$ via the same formula, with stabiliser $G$ at the constant metric. 


Let now $\cA$ be the space of holomorphic $\kg_\C$-valued $(1,0)$-forms in $\D$. We can introduce coordinates on $\cA$ using the coefficients of the power series expansion about 0:
\begin{equation}\label{eq:coeffs}
\alpha(z)=\sum_{b\in\kB,m\geq1}\alpha_{b,m}z^{m-1}.
\end{equation}
Given $H=g^*g\in\cM_\D$, we have by definition $H^{-1}\del H=g^{-1}\del g\in\cA$. Every $\alpha\in\cA$ can be viewed as a flat connection in the trivial bundle $\D\times G_\C$, so there exists a unique $g\in D_0G_\C$ such that $g^{-1}\del g=\alpha$. Hence, we get a one-to-one correspondence between $\cM_\D$ and $\cA$ given by $H\mapsto H^{-1}\del H$. In particular, we can endow $\cM_\D$ with the local uniform topology inherited from $\cA$, together with its Borel $\sigma$-algebra. 

        \subsection{Kac--Moody unitarising measures}\label{subsubsec:KMU}
Let $\kappa<-2\check h$ and $L^\omega\kg_\C$ be the loop algebra, i.e. the Lie algebra of $\kg_\C$-valued holomorphic functions defined in the neighbourhood of $\S^1$. The Kac--Moody unitarising measure $\nu_\kappa$ is a Borel probability measure on $\cA$ which was introduced in \cite{Bav26}. Its precise definition is of secondary importance compared to the integration by parts formula that it satisfies, which is recalled in \eqref{eq:ipp}
 below.
 
 Following the discussion of Section \ref{subsec:HYM}, we can lift $\nu_\kappa$ to a $G$-invariant probability measure on $\cM_\D$ in the following way. Let $\d\gamma$ denote the Haar measure on $G$ and view $\nu_\kappa\otimes\d\gamma$ as a probability measure on $\{g\in DG_\C|\,g(0)\in G\}$ (using the map $D_0G_\C\times G\ni(g,\gamma)\mapsto g\gamma$). We can then consider the random metric $\gamma^*g^*g\gamma=\gamma^{-1}g^*g\gamma$, whose law is invariant under global conjugation by $G$ (this will ensure that the one-point function is proportional to the identity, see Proposition \ref{prop:one-point}). In order to avoid multiple notations, we will henceforth identify $\nu_\kappa$ with this $G$-invariant measure on $\cM_\D$ (this will cause no confusion since the correlation functions introduced in \eqref{eq:def-correls} are invariant under global conjugation).

In \cite[Section~3]{Bav26}, we introduced a commuting pair of representations of the loop algebra $(\cJ_u,\bar\cJ_u)_{u\in L^\omega\kg_\C}$ acting as (unbounded) differential operators on $L^2(\nu_\kappa)$. Given $b\in\kB$ and $m\in\Z$, we use the shorthand $\cJ_{b,m}:=\cJ_{b\otimes z^m}$. The infinitesimal form of \cite[Theorem~1.1]{Bav26} is the following integration by parts formula. Let $\hat D_\infty^\omega\kg_\C\subset L^\omega\kg_\C$ be the Lie algebra of $\kg_\C$-valued holomorphic functions in the neighbourhood of $\hat\C\setminus\D$ which vanish at $\infty$. Then, for all $u\in\hat D^\omega_\infty\kg_\C$ and all $F$ in the domain of $\cJ_u$, we have
\begin{equation}\label{eq:ipp}
\int_{\cM_\D}\cJ_u(F)\d\nu_\kappa=-\kappa\int_{\cM_\D}\alpha(u)F\d\nu_\kappa,
\end{equation}
where $\alpha$ is the differential form on $\cM_\D$ defined by 
\[\alpha_H(u):=\frac{1}{2i\pi}\oint_{r\S^1}\tr(uH^{-1}\del H),\]
and $r\in(0,1)$ is chosen large enough such that $r\S^1$ is contained in the domain of analyticity of $u$. The conjugate formula holds for the representation $(\bar\cJ_u)_{u\in L^\omega\kg_\C}$. We will freely identify $\alpha$ (viewed as a differential form on $\cM_\D$) with the holomorphic form $H^{-1}\del H=g^{-1}\del g\in\cA$.

        \subsection{\texorpdfstring{$\cH_V$-valued measures on the circle}{Positive matrix-valued measures on the circle}}\label{subsubsec:measures}

Let $(V,\langle\cdot,\cdot\rangle_V)$ be a finite dimensional complex Hilbert space and $\cH_V$ be the space of positive definite Hermitian endomorphisms of $V$. Recall that $\cH_V$ comes with a partial order. Let $\mathfrak M(\S^1;\cH_V)$ be the space of positive Borel measures on $\S^1$ with values in $\cH_V$. By definition, an element $M\in\mathfrak M(\S^1;\cH_V)$ is a countably additive non-negative set function on the Borel sets of $\S^1$ with values in $\cH_V$. Note that $M$ is also described by the (finite) family of its matrix coefficients, which are signed measures on $\S^1$. Moreover, the measure $\mathrm{Tr}_V(M)$ is a positive real measure, and all the matrix coefficients are absolutely continuous with respect to it. 

The space $\mathfrak M(\S^1;\cH_V)$ is endowed with the weak topology, by which a sequence $(M^n)_{n\geq0}$ converges to $M$ if $M^n(A)\to M(A)$ for all Borel sets $A\subset\S^1$. Equivalently, $(M^n)_{n\geq0}$ converges weakly if each matrix coefficients converge weakly. An elementary generalisation of Prokhorov's theorem is that the sequence $(M^n)_{n\geq0}$ admits weakly convergent subsequences if the sequence of total masses $\mathrm{Tr}_V(M^n(\S^1))_{n\geq0}$ is bounded. 

Let $(X,\mathcal{X},\nu)$ be a probability space. A random $\cH_V$-valued measure on $\S^1$ is a measurable map $M:X\to\mathfrak M(\S^1;\cH_V)$. We say that a family of random measures $(M^\epsilon)_{\epsilon>0}$ converges weakly in probability to some random measure $M$ if all its matrix coefficients converge weakly in probability in the sense of \cite{Berestycki}. 

\section{Proof of Theorem \ref{thm:convergence}}\label{sec:proof}

In the whole section, we assume fixed $\varrho:G_\C\to\mathrm{GL}(V)$ as in the statement of Theorem~\ref{thm:convergence}. The induced Lie algebra representation is denoted $\rho:\kg_\C\to\mathrm{End}(V)$. Given $z_1,...,z_N\in\D$ pairwise distinct, we introduce the correlation function
\[\left\langle\prod_{j=1}^N\tr(H(z_j))\right\rangle_{\kappa,\rho}:=\int_{\cM_\D}\prod_{j=1}^N\mathrm{Tr}_V\left(\varrho(H(z_j))\right)\d\nu_\kappa(H).\]
In fact, we will only be concerned with the case $N\in\{1,2\}$, and we define:
\begin{equation}\label{eq:def-correls}
\cZ_{\kappa,\rho}(z):=\langle\tr(H(z))\rangle_{\kappa,\rho}\qand\cG_{\kappa,\rho}(z_1,z_2):=\frac{\langle\tr(H(z_1))\tr(H(z_2))\rangle_{\kappa,\rho}}{\cZ_{\kappa,\rho}(z_1)\cZ_{\kappa,\rho}(z_2)}.
\end{equation}



    \subsection{Some variational formulas}\label{subsec:variational-formula}

This section gives some technical formulas which will be used repeatedly. Lemma \ref{lem:var1} is the key result explaining the shift of the level by the dual Coxeter number. In the statement, we fix $z\in\D$ and view $g(z)$ as the evaluation map on $\cM_\D$ sending $H=g^*g$ to $g(z)\in G_\C$.
\begin{lemma}\label{lem:var1}
For all $H=g^*g\in\cM_\D$ and all $z\in\D$, we have the equality in $\mathrm{End}(V)$
    \[\sum_{b\in\kB,m\geq1}z^{m-1}\varrho(H(z))\circ\rho(g^{-1}\cJ_{b,-m}g(z))\circ\rho(b)=\check h\varrho(H(z))\circ\rho(\alpha(z)),\]
    with $\alpha=H^{-1}\del H\in\cA$.
\end{lemma}
In the statement, $\cJ_{b,-m}g(z)$ is a tangent vector above $g(z)\in G_\C$, and $g^{-1}\cJ_{b,-m}g(z)$ is the transport of this vector back to the tangent space at $1_{G_\C}$, to which we can apply $\rho$. For readability in the proof of this lemma, we will treat $G$ as a matrix group and omit the dependence on the choice of representation. With this convention, the quantity we seek to compute can be written more simply as $\sum_{b\in\kB,m\geq1}z^{m-1}g^*(z)\cJ_{b,-m}(g(z))b$. The methodology will be the same as the one used in \cite[Section 5.1]{Bav26}, to which we refer for additional details.

\begin{proof}
 Let $u\in\hat D^\omega_\infty\kg_\C$ and extend it arbitrarily to some smooth function on $\hat\C$. Then, $u$ generates a small motion of $g$ of the form $e^{-\epsilon\tilde u-\bar\epsilon\tilde u^\dagger}ge^{\epsilon u}$ where $\tilde u$ is the solution to 
    \[\delbar\tilde u=\Ad_g(\delbar u)\text{ on }\hat\C\qand\tilde u(\infty)=0,\]
    and $\tilde u^\dagger(z)=\theta(\tilde u(1/\bar z))$.
    Explicitly, we have for all $z$ away from the support of $\delbar u$,
    \begin{equation}\label{eq:tilde-u}
    \tilde u(z)=\frac{i}{2\pi}\oint g(\zeta)u(\zeta)g^{-1}(\zeta)\frac{\d\zeta}{\zeta-z},
    \end{equation}
    where the contour surrounds the support of $\delbar u$ and disconnects $z$ from it. By Cauchy's integral formula and elementary contour deformation, we then get for all $z\in\D$:
    \[u(z)-\Ad_{g^{-1}(z)}(\tilde u(z))=\frac{i}{2\pi}\oint_{r\S^1}\left(u(\zeta)-g^{-1}(z)g(\zeta)u(\zeta)g^{-1}(\zeta)g(z)\right)\frac{\d\zeta}{\zeta-z},\]
    where $r\in(|z|,1)$ is chosen arbitrarily close to $1$, so that $r\S^1$ is contained in the domain of analyticity of $u$. Applying this to the canonical basis, we get for each $b\in\kB$:
    \begin{equation}\label{eq:residue-formula}
    \begin{aligned}
        \sum_{m=1}^\infty\cJ_{b,-m}(g(z))z^{m-1}
        &=\frac{i}{2\pi}\sum_{m=1}^\infty\oint_{r\S^1}g(z)\left(b-g^{-1}(z)g(\zeta)bg^{-1}(\zeta)g(z)\right)\frac{z^{m-1}\d\zeta}{\zeta^m(\zeta-z)}\\
        &=\frac{i}{2\pi}\oint_{r\S^1}g(z)\left(b-g^{-1}(z)g(\zeta)bg^{-1}(\zeta)g(z)\right)\frac{\d\zeta}{(\zeta-z)^2}\\
        &=g(z)[g^{-1}\del g(z),b]=g(z)[\alpha(z),b].
    \end{aligned}
    \end{equation}
In the last line, we applied the residue theorem to the meromorphic form in $\D$ having a unique double pole at $\zeta=z$. Summing over $b\in\kB$ and applying Lemma \ref{lem:coxeter}, we get
\[\sum_{b\in\kB,m\geq1}z^{m-1}g^*(z)\cJ_{b,-m}(g(z))b=\sum_{b\in\kB}H(z)[\alpha(z),b]b=\check hH(z)\alpha(z),\]
ending the proof.
\end{proof}

\begin{lemma}\label{lem:var2}
    For all $H\in\cM_\D$ and all $z\in\D$, we have the equality in $\mathrm{End}(V)$
    \[\sum_{b\in\kB,m\geq1}z^{m-1}\rho(\cJ_{b,-m}(g^*)\theta g(z))\circ\varrho(H(z))\circ\rho(b)=\frac{-\Omega_\rho}{z-\bar z^{-1}}\varrho(H(z)).\]
\end{lemma}
In the statement, $\cJ_{b,-m}(g^*(z))$ is a tangent vector above $g^*(z)\in G_\C$, and the quantity $\cJ_{b,-m}(g^*)\theta g(z)=\cJ_{b,-m}(g^*)g^{*-1}(z)$ is the transport of this vector back to the tangent space at $1_{G_\C}$ (i.e. the Lie algebra $\kg_\C$), to which we can apply $\rho$. As in the previous proof, we will omit the dependency on the representation. With this convention, the quantity we seek to compute is $\sum_{b\in\kB,m\geq1}z^{m-1}\cJ_{b,-m}(g^*(z))g(z)b$.

\begin{proof}
 With the notations from the previous proof, the small variation of $g^*$ is of the form $e^{\bar\epsilon u^*}g^*e^{-\bar\epsilon\tilde u^*-\epsilon\tilde u^{\dagger*}}$. Note that $\tilde u^{\dagger*}(z)=-\tilde u(1/\bar z)$, which is antiholomorphic in $\D$ since $\tilde u$ is holomorphic in $\D^*$ (recall \eqref{eq:tilde-u} for its definition). Applying this to the canonical basis and summing over $m\geq1$ for fixed $b\in\kB$, we get


    \begin{equation}\label{eq:var-g-star}
    \begin{aligned}
    \sum_{m=1}^\infty\cJ_{b,-m}(g^*(z))z^{m-1}
    &=\frac{i}{2\pi}\sum_{m=1}^\infty\oint_{\S^1}g^*(z)g(\zeta)bg^{-1}(\zeta)\frac{z^{m-1}\d\zeta}{\zeta^m(\zeta-\bar z^{-1})}\\
    &=\frac{i}{2\pi}\oint_{\S^1} g^*(z)g(\zeta)bg^{-1}(\zeta)\frac{\d\zeta}{(\zeta-z)(\zeta-\bar z^{-1})}\\
    &=\frac{-1}{z-\bar z^{-1}}H(z)bg^{-1}(z).
    \end{aligned}
    \end{equation}
We applied the residue theorem in the last line. Summing over $b\in\kB$ yields
\[\sum_{b\in\kB,m\geq1}\cJ_{b,-m}(g^*(z))g(z)b=\frac{-1}{z-\bar z^{-1}}\sum_{b\in\kB}H(z)bb=\frac{-\Omega_\rho}{z-\bar z^{-1}}H(z),\]
ending the proof.
\end{proof}

    \subsection{One- and two-point functions}\label{subsec:correlations}
In this section, we prove exact formulas and bounds for the correlation functions defined in \eqref{eq:def-correls}.

The next proposition shows that the one-point function is a power of the Poincar\'e metric in $\D$, signalling the conformal invariance of the theory. As mentioned in the introduction, the exponent matches the conformal weights of the WZW model.

\begin{proposition}\label{prop:one-point}
    For all $z\in\D$, we have the equality in $\mathrm{End}(V)$
    \[\int_{\cM_\D}\varrho(H(z))\d\nu_\kappa(H)=(1-|z|^2)^\frac{\Omega_\rho}{\kappa+\check h}\mathrm{Id}_V.\]
\end{proposition}

\begin{proof}
By construction, the law of $H(z)$ is invariant under conjugation by elements of $G$, hence the matrix $\int_{\cM_\D}\varrho(H(z))\d\nu_\kappa(H)$ commutes with $\varrho(G)$. It follows from Schur's lemma and the irreducibility of $\varrho$ that this matrix is proportional to $\mathrm{Id}_V$. Since $\varrho(H(0))=\mathrm{Id}_V$ almost surely, taking the trace identifies the constant with $\frac{\cZ_{\kappa,\rho}(z)}{\cZ_{\kappa,\rho}(0)}$. We will compute $\cZ_{\kappa,\rho}$ by showing that it solves a differential equation.

    Assume for a moment that $\cZ_{\kappa,\rho}$ is differentiable and that we can exchange the differential and the expectation, so that $\del_z\cZ_{\kappa,\rho}(z)=\langle\del_z\tr(H(z))\rangle_{\kappa,\rho}$ (we will justify this below). We write the expansion $H^{-1}\del H(z)=\sum_{b\in\kB,m\geq1}\alpha_{b,m}b\otimes z^{m-1}$ and use the integration by parts formula \eqref{eq:ipp} to get
    \begin{align*}
    \kappa\del_z\cZ_{\kappa,\rho}(z)
    &=\kappa\int_{\cM_\D}\mathrm{Tr}_V(\varrho(H(z))\circ\rho(\alpha(z)))\d\nu_\kappa(H)\\
&=\kappa\sum_{b\in\kB,m\geq1}z^{m-1}\int_{\cM_\D}\alpha_{b,m}\mathrm{Tr}_V(\varrho(H(z))\circ\rho(b))\d\nu_\kappa(H)\\
    &=-\sum_{b\in\kB,m\geq1}z^{m-1}\int_{\cM_\D}\cJ_{b,-m}\left(\mathrm{Tr}_V(\varrho(H(z))\circ\rho(b))\right)\d\nu_\kappa(H).
    \end{align*}
Recall that $H=g^*g$, so by the Leibniz rule the last line produces one contribution from $\cJ_{b,-m}(g(z))$ and one contribution from $\cJ_{b,-m}(g^*(z))$, which are exactly the quantities computed in Lemma \ref{lem:var1} and Lemma \ref{lem:var2} respectively (and also justifies the convergence of the above series). Hence,
\begin{align*}
    \kappa\del_z\cZ_{\kappa,\rho}(z)=-\check h\int_{\cM_\D}\mathrm{Tr}_V(\varrho(H(z))\circ\rho(\alpha(z)))\d\nu_\kappa(H)+\frac{\Omega_\rho}{z-\bar z^{-1}}\int_{\cM_\D}\mathrm{Tr}_V(\varrho(H(z)))\d\nu_\kappa(H).
\end{align*}   
    Note that the first term on the right-hand side equals $-\check h\del_z\cZ_{\kappa,\rho}(z)$, while the second one is $\frac{\Omega_\rho}{z-\bar z^{-1}}\cZ_{\kappa,\rho}(z)$. Rearranging gives
    \begin{equation}\label{eq:diff-one-point}
    \del_z\cZ_{\kappa,\rho}(z)=\frac{\Omega_\rho}{\kappa+\check h}\frac{\bar z}{|z|^2-1}\cZ_{\kappa,\rho}(z).
    \end{equation}
    
    By symmetry, we get the conjugate formula for $\del_{\bar z}\cZ_{\kappa,\rho}$. We now justify the differentiability of $\cZ_{\kappa,\rho}$. Let $\varepsilon\in\R$ small and write using Fubini $\frac{1}{\varepsilon}(\cZ_{\kappa,\rho}(z+\varepsilon)-\cZ_{\kappa,\rho}(z))=\frac{1}{\varepsilon}\int_0^\varepsilon\langle\del_t\tr(H(z+t))\rangle_{\kappa,\rho}\d t$. From the previous paragraph, the integrand in the right-hand side is just a multiple of $\cZ_{\kappa,\rho}(z+t)$. Hence the right-hand side is differentiable at $\varepsilon=0$, i.e. $\cZ_{\kappa,\rho}$ is differentiable in the real direction at $z$. The same argument shows that $\cZ_{\kappa,\rho}$ is also differentiable in the imaginary direction, and the expression for the differential is indeed the one given by~\eqref{eq:diff-one-point}.
    
    It follows from \eqref{eq:diff-one-point} and the conjugate formula that $(z\del_z-\bar z\del_{\bar z})\cZ_{\kappa,\rho}=0$, which is nothing but the rotational invariance of $\cZ_{\kappa,\rho}$. As for the radial part, we get in polar coordinates $z=re^{i\theta}$
    \[r\del_r\cZ_{\kappa,\rho}(re^{i\theta})=(z\del_z+\bar z\del_{\bar z})\cZ_{\kappa,\rho}(z)=\frac{\Omega_\rho}{\kappa+\check h}\frac{2r^2}{r^2-1}\cZ_{\kappa,\rho}(re^{i\theta}),\]
    which can be rewritten as
    \[\del_r\log\cZ_{\kappa,\rho}(r)=\frac{\Omega_\rho}{\kappa+\check h}\frac{2r}{r^2-1}=\frac{\Omega_\rho}{\kappa+\check h}\del_r\log(1-r^2).\]
    Integrating this equation gives $\cZ_{\kappa,\rho}(z)=(1-|z|^2)^\frac{\Omega_\rho}{\kappa+\check h}\cZ_{\kappa,\rho}(0)$ as required.
\end{proof}

\begin{proposition}\label{prop:two-point}
    For all $z_1,z_2\in\D$ distinct,
    \begin{equation}\label{eq:bound}
    \left|1-z_1\bar z_2\right|^\frac{2\lambda_\rho^-}{\kappa+\check h}\leq\cG_{\kappa,\rho}(z_1,z_2)\leq\left|1-z_1\bar z_2\right|^\frac{2\lambda_\rho^+}{\kappa+\check h},
    \end{equation}
    where $\lambda_\rho^\pm$ are defined in \eqref{eq:lambda_rho}. Moreover, if $\frac{2\lambda_\rho^+}{\kappa+\check h}>-1$, the family $(\cG_{\kappa,\rho}(e^{-\epsilon}\,\cdot,e^{-\delta}\,\cdot)|_{\S^1\times\S^1})_{\epsilon,\delta>0}$ converges almost everywhere and in $L^1(\S^1\times\S^1)$ as $\epsilon,\delta\to0$.
\end{proposition}

\begin{proof}
    As in the proof of Proposition \ref{prop:one-point}, we will study the differential of $\cG_{\kappa,\rho}$. We omit the proof of differentiability of $\cG_{\kappa,\rho}$ as it is identical to that of $\cZ_{\kappa,\rho}$. Using the integration by parts formula \eqref{eq:ipp} and the Leibniz rule, we have (writing again $H^{-1}\del H(z)=\sum_{b\in\kB,m\geq1}\alpha_{b,m}b\otimes z^{m-1}$)
    \begin{equation}\label{eq:two-point-start}
    \begin{aligned}
    &\kappa\del_{z_1}\langle \tr(H(z_1))\tr(H(z_2))\rangle_{\kappa,\rho}\\
&\quad=\kappa\int_{\cM_\D}\mathrm{Tr}_V(\varrho(H(z_1))\circ\rho(\alpha(z_1)))\mathrm{Tr}_V(\varrho(H(z_2)))\d\nu_\kappa(H)\\
&\quad=-\sum_{b\in\kB,m\geq1}z_1^{m-1}\int_{\cM_\D}\cJ_{b,-m}\left(\mathrm{Tr}_V(\varrho(H(z_1))\circ\rho(b))\right)\mathrm{Tr}_V(\varrho(H(z_2)))\d\nu_\kappa(H)\\
&\qquad\;-\sum_{b\in\kB,m\geq1}z_1^{m-1}\int_{\cM_\D}\mathrm{Tr}_V(\varrho(H(z_1))\circ\rho(b))\cJ_{b,-m}\left(\mathrm{Tr}_V(\varrho(H(z_2)))\right)\d\nu_\kappa(H).
    \end{aligned}
    \end{equation}
The first sum in the last equality is already known by the previous lemma, so we focus on the second one. By the Leibniz rule, this second term itself can be divided in two contributions: one coming from $\cJ_{b,-m}(g(z_2))$, and the other one from $\cJ_{b,-m}(g^*(z_2))$. 

\textbf{Contribution of $\cJ_{b,-m}(g(z_2))$}. Here we make implicit the dependency on $\varrho$ as in the proof of Lemma \ref{lem:var1}. By a straightforward adaptation of the beginning of the proof of Lemma~\ref{lem:var1} leading to \eqref{eq:residue-formula}, we have for each $b\in\kB$
\[\sum_{m=1}^\infty\cJ_{b,-m}(g(z_2))z_1^{m-1}=\frac{g(z_2)bg^{-1}(z_2)-g(z_1)bg^{-1}(z_1)}{z_2-z_1}g(z_2).\]
Summing over $b\in\kB$, using the cyclicity of the trace and the fact that $\Ad_{g(z_1)}(\kB)$ is an orthonormal basis of $\kB$, we then get
\begin{equation}\label{eq:bar}
\begin{aligned}
&\sum_{b\in\kB,m\geq1}\tr(H(z_1)b)\tr(g^*(z_2)\cJ_{b,-m}(g(z_2)))z_1^{m-1}\\
&\quad=\frac{-1}{z_1-z_2}\sum_{b\in\kB}\tr(H(z_1)b)\tr(H(z_2)(b-g^{-1}(z_2)g(z_1)bg^{-1}(z_1)g(z_2)))\\
&\quad=\frac{-1}{z_1-z_2}\sum_{b\in\kB}\tr(H(z_1)b)\tr(H(z_2)b)-\tr(H(z_1)g^{-1}(z_1)bg(z_1))\tr(H(z_2)g^{-1}(z_2)bg(z_2))\\
&\quad=\frac{-1}{z_1-z_2}\sum_{b\in\kB}\tr(H(z_1)b)\tr(H(z_2)b)-\tr(g(z_1)g^*(z_1)b)\tr(g(z_2)g^*(z_2)b).
\end{aligned}
\end{equation}
Recall that $\tr(H(z_1)b)$ means $\mathrm{Tr}_V(\varrho(H(z_1)\circ\rho(b))$ and likewise for the other terms. We will now give an expression for the integral of \eqref{eq:bar} over $\nu_\kappa$. First, recall that the measure $\nu_\kappa$ is initially defined on $\cA$, and this measure is invariant under complex conjugation of the Taylor coefficients introduced in \eqref{eq:coeffs} (this is immediate from the fact that the representations $(\bJ_u)_{u\in L^\omega\kg_\C}$ and $(\bar\bJ_u)_{u\in L^\omega\kg_\C}$ from \cite[Section 3]{Bav26} play interchangeable roles). Hence, the law of the holomorphic function $g:\D\to G_\C$ is equal to the law of $z\mapsto g^*(\bar z)$, which implies that the law of $gg^*$ equals the law of $z\mapsto H(\bar z)$. Integrating \eqref{eq:bar} over $\nu_\kappa$, we get 
\begin{equation}\label{eq:barbis}
\begin{aligned}
    &\int_{\cM_\D}\sum_{b\in\kB,m\geq1}z_1^{m-1}\tr(H(z_1)b)\tr(g^*(z_2)\cJ_{b,-m}(g(z_2)))\d\nu_\kappa(H)\\
    &=\frac{-1}{z_1-z_2}\int_{\cM_\D}\mathrm{Tr}_{V\otimes V}\left(\left(\varrho(H(z_1))\otimes\varrho(H(z_2))-\varrho(H(\bar z_1))\otimes\varrho(H(\bar z_2))\right)\circ\sum_{b\in\kB}\rho(b)\otimes\rho(b)\right),
\end{aligned}
\end{equation}
where we used the elementary equality $\mathrm{Tr}_{V\otimes V}(A\otimes B)=\mathrm{Tr}_V(A)\mathrm{Tr}_V(B)$ for $A,B\in\mathrm{End}(V)$. We will see below that the expectation of this last quantity vanishes.

\smallskip

\textbf{Contribution of $\cJ_{b,-m}(g^*(z_2))$}. A straightforward adaptation of Lemma~\ref{lem:var2} gives 
\begin{align*}
&\sum_{b\in\kB,m\geq1}\mathrm{Tr}_V\left(\varrho(H(z_1))\rho(b)\right)\mathrm{Tr}_V\left(\cJ_{b,-m}(\varrho(g^*(z_2)))\varrho(g(z_2))\rho(b)\right)z_1^{m-1}\\
&\quad=\frac{-1}{z_1-\bar z_2^{-1}}\sum_{b\in\kB}\mathrm{Tr}_V\left(\varrho(H(z_1))\circ\rho(b)\right)\mathrm{Tr}_V\left(\varrho(H(z_2))\circ\rho(b)\right)\\
    &\quad=\frac{-1}{z_1-\bar z_2^{-1}}\mathrm{Tr}_{V\otimes V}\left(\left(\varrho(H(z_1))\otimes\varrho(H(z_2))\right)\circ\sum_{b\in\kB}\rho(b)\otimes\rho(b)\right).
\end{align*}

\smallskip

We will now gather all the contributions. To this end, it will be convenient to use the parametrisation $z_1=z=re^{i\theta}$, and $z_2=z\zeta$ with $\zeta\in\bar\D\setminus\{1\}$. We fix $\zeta$ and consider $\cG_{\kappa,\rho}(z,z\zeta)$ as a function of $z$ only. By rotation invariance, we get $(z\del_z-\bar z\del_{\bar z})\cG_{\kappa,\rho}(z,z\zeta)=0$. It remains to compute the radial variation. First, observe that $\cG_{\kappa,\rho}(z_1,z_2)=\cG_{\kappa,\rho}(\bar z_1,\bar z_2)$ by the cyclicity of the trace and the fact that $gg^*$ is equal in law to $z\mapsto H(\bar z)$. Hence, $(z\del_z+\bar z\del_{\bar z})\cG_{\kappa,\rho}(z,z\zeta)$ is invariant under $z\mapsto\bar z$. Now, all the contributions that we have computed above are invariant under this involution, except the one coming from \eqref{eq:barbis} which is antisymmetric and must therefore vanish. Moreover, the fact that we have divided by $\cZ_{\kappa,\rho}(z)\cZ_{\kappa,\rho}(z\zeta)$ in the definition of $\cG_{\kappa,\rho}(z,z\zeta)$ (see \eqref{eq:def-correls}) results in the cancellation of the contributions of the divergent term $\frac{1}{z-\bar z^{-1}}$. Altogether, we get
\begin{equation}\label{eq:edo}
\begin{aligned}
    &(\kappa+\check h)r\del_r\log\cG_{\kappa,\rho}(re^{i\theta},re^{i\theta}\zeta))\\
    &\quad=(\kappa+\check h)(z\del_z+\bar z\del_{\bar z})\left(\log\langle\tr(H(z))\tr(H(z\zeta))\rangle_{\kappa,\rho}-\log\cZ_{\kappa,\rho}(z)-\log\cZ_{\kappa,\rho}(z\zeta)\right)\\
    &\quad=4\Re\left(\frac{\zeta|z|^2}{\zeta|z|^2-1}\right)\frac{\int_{\cM_\D}\mathrm{Tr}_{V\otimes V}\left(\left(\varrho(H(z))\otimes\varrho(H(z\zeta))\right)\circ\sum_{b\in\kB}\rho(b)\otimes\rho(b)\right)\d\nu_\kappa(H)}{\langle\tr(H(z))\tr(H(z\zeta))\rangle_{\kappa,\rho}}\\
\end{aligned}
\end{equation}
The prefactor $4\Re(\frac{\zeta|z|^2}{\zeta|z|^2-1})$ comes from the fact that the operator $z\del_z$ hits both variables of $\cG_{\kappa,\rho}$, producing $\frac{z}{z-\bar z^{-1}\bar\zeta^{-1}}+\frac{z\zeta}{z\zeta-\bar z^{-1}}=2\Re(\frac{|z|^2\zeta}{|z|^2\zeta-1})$. We get the same contribution from $\bar z\del_{\bar z}$. From \eqref{eq:lambda_rho}, the fraction in the last line is bounded by the spectrum of the operator $\sum_{b\in\kB}\rho(b)\otimes\rho(b)$. Observing also that $2\Re(\frac{r\zeta}{r^2\zeta-1})=\del_r\log|1-r^2\zeta|$ and that $\cG_{\kappa,\rho}(z_1,z_2)\to1$ as $z_1,z_2\to0$, the previous equation integrates to the inequality \eqref{eq:bound}, which is the first part of Proposition \ref{prop:two-point}. 

For the second part, the previous bound shows that the family $(\cG_{\kappa,\rho}(e^{-\epsilon}\,\cdot,e^{-\delta}\,\cdot)|_{\S^1\times\S^1})_{\epsilon,\delta>0}$ is dominated by the integrable function $\S^1\times\S^1\ni(z_1,z_2)\mapsto|z_1-z_2|^\frac{2\lambda_\rho^+}{\kappa+\check h}$ (this is where the condition $\frac{2\lambda_\rho^+}{\kappa+\check h}>-1$ is used), and that the right-hand side of \eqref{eq:edo} is uniformly bounded for fixed $\zeta\in\bar\D\setminus\{1\}$. Hence, by dominated convergence, it suffices to prove pointwise convergence away from the diagonal, which is immediate by integrating the differential equation \eqref{eq:edo}.
\end{proof}

For the next result, we introduce the modified correlation function
\[\tilde\cG_{\kappa,\rho}(z_1,z_2):=\frac{\langle\tr(H(z_1)H(z_2))\rangle_{\kappa,\rho}}{\cZ_{\kappa,\rho}(z_1)\cZ_{\kappa,\rho}(z_2)}.\]
\begin{proposition}\label{prop:two-point-tilde}
    For $\kappa<\min\{-2\check h,-\check h-2\lambda_\rho^+\}$, the family $(\tilde\cG_{\kappa,\rho}(e^{-\epsilon}\,\cdot,e^{-\delta}\,\cdot)|_{\S^1\times\S^1})_{\epsilon,\delta>0}$ converges almost everywhere and in $L^1(\S^1\times\S^1)$ as $\epsilon,\delta\to0$.
\end{proposition}
\begin{proof}
    By the elementary inequality $\mathrm{Tr}_V(H_1H_2)\leq\mathrm{Tr}_V(H_1)\mathrm{Tr}_V(H_2)$ valid for $H_1,H_2\in\cH_V$, we have $\tilde\cG_{\kappa,\rho}(z_1,z_2)\leq\cG_{\kappa,\rho}(z_1,z_2)$. Hence, by the first part of Proposition \ref{prop:two-point} and the dominated convergence theorem, it suffices to prove convergence almost everywhere. This can be achieved in exactly the same way as in the proof of Proposition \ref{prop:two-point}: we use Lemmas~\ref{lem:var1} and \ref{lem:var2} to write a differential equation for $\tilde\cG_{\kappa,\rho}$ and observe that all the terms are regular except on the diagonal. Hence, we can integrate the equation to get pointwise limits on the circle away from the diagonal. We omit the details of these computations.
\end{proof}

    \subsection{Completing the proof}\label{subsec:proof}
We will now conclude the proof of Theorem \ref{thm:convergence}, which is a routine argument given Proposition \ref{prop:two-point-tilde}. We will follow closely \cite[Section 6]{Berestycki} and we will omit some details. 

 First, we show that for each Borel set $A\subset\S^1$, the family of $\cH_V$-valued random variables $(M^\epsilon_{\kappa,\rho}(A))_{\epsilon>0}$ converges in $L^2(\nu_\kappa)$, with $\mathrm{End}(V)$ being endowed with the Hilbert--Schmidt norm. We have for all $\epsilon,\delta>0$:
\begin{align*}
&\int_{\cM_\D}\mathrm{Tr}_V\left(\left(M^\epsilon_{\kappa,\rho}(A)-M^\delta_{\kappa,\rho}(A)\right)^2\right)\d\nu_\kappa\\
     &=\int_{A\times A}\left(\tilde\cG_{\kappa,\rho}(e^{-\epsilon+i\theta_1},e^{-\epsilon+i\theta_2})-2\tilde\cG_{\kappa,\rho}(e^{-\epsilon+i\theta_1},e^{-\delta+i\theta_2})+\tilde\cG_{\kappa,\rho}(e^{-\delta+i\theta_1},e^{-\delta+i\theta_1})\right)\d\theta_1\d\theta_2\\
    &\underset{\epsilon,\delta\to0}{\to}0,
\end{align*}
with the last line following from Proposition \ref{prop:two-point-tilde}. Hence, the sequence $(M^\epsilon_{\kappa,\rho}(A))_{\epsilon>0}$ is Cauchy in $L^2(\nu_\kappa;\cH_V)$, hence converges to some random variable which we denote $M_{\kappa,\rho}(A)$. In particular, the family of (real positive) random variables $\mathrm{Tr}_V(M^\epsilon_{\kappa,\rho}(\S^1))_{\epsilon>0}$ converges in $\nu_\kappa$-probability, and so do the total variations of the matrix coefficients of $M^\epsilon_{\kappa,\rho}$. 

 Now, let $\mathcal{I}$ be the collection of open arcs in $\S^1$ with rational endpoints ($\mathcal I$ is a countable generator of the Borel $\sigma$-algebra). From the previous paragraph and a diagonal argument, we can find a decreasing sequence $\epsilon_n\to0$ such that $(M^{\epsilon_n}_{\kappa,\rho}(A))_{n\geq0}$ converges to $M_{\kappa,\rho}(A)$ as $n\to\infty$, simultaneously for $\nu_\kappa$-a.e. $H\in\cM_\D$ and all $A\in\mathcal{I}$. Since the total masses are $\nu_\kappa$-a.s. uniformly bounded, the above convergence extends to all Borel sets $A\subset\S^1$.

Finally, as discussed in Section \ref{subsubsec:measures}, the uniform boundedness of the total masses implies that $\nu_\kappa$-a.s. we can extract a further subsequence $\tilde\epsilon_n$ such that each matrix coefficient of $(M^{\tilde\epsilon_n}_{\kappa,\rho})_{n\geq0}$ converges weakly. Following \cite[Section 6]{Berestycki}, one can deduce that every limiting subsequence $\tilde M_{\kappa,\rho}$ must satisfy $\tilde M_{\kappa,\rho}(A)=M_{\kappa,\rho}(A)$ for all Borel sets $A\subset\S^1$, which implies that $\nu_\kappa$-a.s. the sequence $(M^{\epsilon_n}_{\kappa,\rho})_{n\geq0}$ converges weakly to $M_{\kappa,\rho}$. As explained in \cite[Section 6]{Berestycki}, this implies the weak convergence in probability of $(M^\epsilon_{\kappa,\rho})_{\epsilon>0}$ to $M_{\kappa,\rho}$.

\bibliographystyle{alpha}
\bibliography{yang-mills}

\end{document}